\documentclass[12pt,reqno]{amsart}

\usepackage{mathtools}

\usepackage{mathdots}

\usepackage{color}

\usepackage{pdfsync}

\usepackage{enumitem}

\usepackage{wasysym}

\usepackage{amssymb}

\usepackage{mathrsfs}

\DeclareMathAlphabet{\mathpzc}{OT1}{pzc}{m}{it}

\usepackage[all]{xy}

\usepackage{tikz}
\usetikzlibrary{arrows,matrix,decorations.pathmorphing,decorations.pathreplacing,positioning,shapes.geometric,shapes.misc,decorations.markings,decorations.fractals,calc,patterns}

\usepackage{graphicx}

\usepackage{float}

\usepackage{hyperref}

\usepackage[bottom]{footmisc}

\usepackage{moreenum}

\usepackage{makecell}

\entrymodifiers={+!!<0pt,\fontdimen22\textfont2>}

\def\cP{\mathscr{P}}

\def\BN{\mathbb{N}}

\def\BR{\mathbb{R}}

\def\BZ{\mathbb{Z}}

\def\adots{\mathinner{\mkern1mu\raise1.0pt\vbox{\kern7.0pt\hbox{.}}\mkern2mu\raise5.0pt\hbox{.}\mkern2mu\raise9.0pt\hbox{.}\mkern1mu}}

\newcommand{\cupdot}{\mathbin{\mathaccent\cdot\cup}}

\def\diag{\operatorname{diag}}
\def\dddots{\mathinner{\mkern1mu\raise10.0pt\vbox{\kern7.0pt\hbox{.}}\mkern2mu\raise5.3pt\hbox{.}\mkern2mu\raise1.0pt\hbox{.}\mkern1mu}}
\def\dddotssmall{\mathinner{\mkern1mu\raise7.0pt\vbox{\kern7.0pt\hbox{.}}\mkern-1mu\raise4pt\hbox{.}\mkern-1mu\raise1.0pt\hbox{.}\mkern1mu}}

\def\K{\operatorname{K}}
\def\K0{\operatorname{K}_0}

\def\PSL2{\operatorname{PSL}_2}

\def\Rpos{\BR_{>0}}

\def\SL2{\operatorname{SL}_2}

\newcommand\TPath[4]{\cP(#2,#3,#4)}

\numberwithin{equation}{section}

\newtheorem{Lemma}{Lemma}[section]

\theoremstyle{definition}
\newtheorem{Definition}[Lemma]{Definition}
\newtheorem{Setup}[Lemma]{Setup}

\newtheorem{ThmIntro}{Theorem}

\newtheorem*{bfhpg*}{}

  {\begin{list}{}{%
    \settowidth{\labelwidth}{\textbf{#1:}}%
    \setlength{\leftmargin}{\labelwidth}\addtolength{\leftmargin}{\labelsep}}}%
  {\end{list}}

\newcommand{\Canakci}{\c{C}anak\c{c}\i}
\newcommand{\Ilke}{\.{I}lke }

\begin{document}

\setlength{\parindent}{0pt}
\setlength{\parskip}{7pt}

\title[Friezes, weak friezes, and T-paths]{Friezes, weak friezes, and T-paths}

\author{\Ilke \Canakci}

\address{Department of Mathematics, VU University, Faculty of Sciences, De Boelelaan 1111, 1081 HV Amsterdam, The Netherlands}

\email{i.canakci@vu.nl}

\author{Peter J\o rgensen}

\address{School of Mathematics, Statistics and Physics, Newcastle University, Newcastle upon Tyne NE1 7RU, United Kingdom}
\email{peter.jorgensen@ncl.ac.uk}

\urladdr{http://www.staff.ncl.ac.uk/peter.jorgensen}





\begin{abstract} 

\medskip
\noindent
Frieze patterns form a nexus between algebra, combinatorics, and geometry.  $T$-paths with respect to triangulations of surfaces have been used to obtain expansion formulae for cluster variables.

\medskip
\noindent
This paper will introduce the concepts of weak friezes and $T$-paths with respect to dissections of polygons.  Our main result is that weak friezes are characterised by satisfying an expansion formula which we call the $T$-path formula.

\medskip
\noindent
We also show that weak friezes can be glued together, and that the resulting weak frieze is a frieze if and only if so was each of the weak friezes being glued.

\end{abstract}

\maketitle

\setcounter{section}{-1}
\section{Introduction}
\label{sec:introduction}

This paper will introduce weak friezes and $T$-paths with respect to dissections of polygons, and show that weak friezes are characterised by satisfying an expansion formula which we call the $T$-path formula.  See Definitions \ref{def:friezes_and_weak_friezes}, \ref{def:T-paths_and_T-path_formula} and Theorem \ref{thm:A}.

Weak friezes are strongly related to the generalised frieze patterns defined in \cite[sec.\ 5]{BHJ} and to the frieze patterns of \cite[sec.\ 1]{Cox}, which form a nexus between algebra, combinatorics, and geometry; see \cite{M} for a recent survey.  $T$-paths with respect to dissections of polygons are a generalisation of $T$-paths with respect to triangulations of polygons as defined in \cite[sec.\ 1.2]{Sch1}.  

In preparation for the proof of Theorem \ref{thm:A}, we show in Theorem \ref{thm:B} that weak friezes can be glued together.  There is also a notion of frieze, and Theorem \ref{thm:C} shows that when weak friezes are glued together, the result is a frieze if and only if each of the weak friezes being glued is a frieze.

Note that $T$-paths with respect to triangulations of polygons and general surfaces were used in \cite[thm.\ 2.10]{BR}, \cite[thm.\ 3.8]{GM}, \cite[thm.\ 1.1]{MSW}, \cite[thm.\ 1.2]{Sch1}, \cite[thm.\ 3.1]{Sch2}, and \cite[thm.\ 3.5]{ST} to obtain expansion formulae for cluster variables.  This permitted the resolution of the positivity conjecture for cluster algebras arising from surfaces, stating that all Laurent polynomials for cluster variables have positive coefficients.  Cluster algebras and the positivity conjecture were introduced in \cite[def.\ 2.3]{FZ} and \cite[sec.\ 1]{FZ}, and the conjecture was later resolved in general, see \cite[cor.\ 0.4]{GHKK} and \cite[thm.\ 1.1]{LS}.

\smallskip
\subsection{Friezes and weak friezes}~\\
\label{subsec:friezes_and_weak_friezes}
\vspace{-3.4ex}

If $P$ is a polygon, $\alpha \neq \beta$ vertices, then there is a diagonal $\{ \alpha,\beta \}$, and the set of diagonals of $P$ is denoted $\diag(P)$.  A dissection $D$ of $P$ is a set of pairwise non-crossing diagonals between non-neighbouring vertices, and the empty dissection $D = \emptyset$ is allowed.

Throughout the paper, $( K,+,\cdot )$ is a fixed semifield, see Definition \ref{def:semifield}.  An important feature is that $K$ has no ``subtraction'' in the sense of an inverse operation of $+$.  Examples are $( \Rpos,+,\cdot )$, the positive real numbers with addition and multiplication, and $( \BZ,\max,+ )$, the so-called tropical semifield.  In the latter case, the operations denoted $+$ and $\cdot$ in the abstract are indeed given by maximum and addition in $\BZ$.
\begin{Definition}
[Friezes and weak friezes]
\label{def:friezes_and_weak_friezes}
Let $P$ be a polygon, $f : \diag(P) \xrightarrow{} K$ a map.  Write $f( \alpha,\beta ) := f\big( \{ \alpha,\beta \}\big)$.
\begin{enumerate}
\setlength\itemsep{4pt}

  \item  $f$ is called a {\em frieze} if it satisfies the {\em Ptolemy relation}
\begin{equation}
\label{equ:Ptolemy}
  f( \alpha,\beta )f( \gamma,\delta ) =
  f( \alpha,\gamma )f( \beta,\delta )
  + f( \alpha,\delta )f( \beta,\gamma )
\end{equation}
when $\{ \alpha,\beta \}$ and $\{ \gamma,\delta \}$ are crossing diagonals.  

\item  If $D$ is a dissection of $P$, then
$f$ is called a {\em weak frieze with respect to $D$} if it satisfies the Ptolemy relation   \eqref{equ:Ptolemy} when $\{ \alpha,\beta \}$ and $\{ \gamma,\delta \}$ are crossing diagonals  satisfying $\{ \gamma,\delta \} \in D$.  

\end{enumerate}
\end{Definition}
The relation between (weak) friezes and (generalised) frieze patterns will be explained in Section \ref{subsec:frieze_patterns}.  Note that friezes with values in the tropical semifield $( \BZ,\max,+ )$ were defined in \cite[sec.\ 2.1]{G}; they are known as tropical friezes.  

Any map is a weak frieze with respect to the empty dissection.  For example, a {\it trivial weak frieze} is a map on diagonals with constant value $1_K$.  Gluing together three trivial weak friezes with values in $\Rpos$, two of them on $4$-gons, the third on a pentagon, gives the weak frieze $f$ on a $9$-gon shown in Figure \ref{fig:9-gon}.  Weak friezes can be glued by Theorem \ref{thm:B}, and the edges along which we glue become internal diagonals which are adjoined to the dissection $D$.

\smallskip
\subsection{$T$-paths}~\\
\label{subsec:T-paths}
\vspace{-3.4ex}

The notion of $T$-path with respect to a triangulation of a polygon was defined in \cite[sec.\ 1.2]{Sch1}, and we generalise it to $T$-path with respect to a dissection as follows.  
\begin{Definition}
[$T$-paths and the $T$-path formula]
\label{def:T-paths_and_T-path_formula}
Let $P$ be a polygon, $D$ a dissection of $P$.  If $\pi_1 \neq \pi_p$ are vertices, then a {\em $T$-path from $\pi_1$ to $\pi_p$ with respect to $D$} is an ordered tuple $\pi = ( \pi_1, \ldots, \pi_p )$ of vertices satisfying the following.
\begin{enumerate}
\setlength\itemsep{4pt}

  \item  $\{ \pi_1,\pi_2 \},$ $\{ \pi_2,\pi_3 \},$ $\ldots,$ $\{ \pi_{ p-1 },\pi_p \}$ are pairwise different diagonals.
  
  \item  No diagonal $\{ \pi_i,\pi_{ i+1 } \}$ crosses a diagonal in $D$.
  
  \item  The diagonals $\{ \pi_{ 2j },\pi_{ 2j+1 } \}$ are in $D$, and cross the diagonal $\{ \pi_1,\pi_p \}$ at pairwise different points which progress monotonically in the direction from $\pi_1$ to $\pi_p$.  

\end{enumerate}
The set of $T$-paths from $\pi_1$ to $\pi_p$ with respect to $D$ is denoted $\TPath{P}{D}{\pi_1}{\pi_p}$.  A map $f : \diag( P ) \xrightarrow{} K$ is said to {\em satisfy the $T$-path formula with respect to $D$} if
\begin{equation}
\label{equ:the_T-path_formula}
  f( \alpha,\beta ) = \sum_{ \pi \in \TPath{P}{D}{\alpha}{\beta} } f( \pi )
\end{equation}
for all vertices $\alpha \neq \beta$ of $P$, where
\begin{equation}
\label{equ:f_on_T-paths}
  f( \pi ) := \frac{ \prod_{\mbox{$i$ odd}}f( \pi_i,\pi_{ i+1 } ) }{ \prod_{\mbox{$j$ even}} f( \pi_j,\pi_{ j+1 } ) }
\end{equation}
for a $T$-path $\pi$. 
\end{Definition}
For example, in the situation of Figure \ref{fig:9-gon}, the $T$-paths from $\alpha$ to $\beta$ with respect to $D$ are shown in Figure \ref{fig:T-paths}.

Our first main result is the following, which is inspired by the expansion formula for cluster algebras of type $A$ given in \cite[thm.\ 1.2]{Sch1}, and has that result as a special case.

\begin{ThmIntro}
[Weak friezes and $T$-paths]
\label{thm:A}
Let $P$ be a polygon with a dissection $D$.  A map $f : \diag(P) \xrightarrow{} K$ is a weak frieze with respect to $D$ if and only if it  satisfies the $T$-path formula with respect to $D$, see Definition \ref{def:T-paths_and_T-path_formula}.
\end{ThmIntro}

For example, the weak frieze $f$ in Figure \ref{fig:9-gon} has value $1$ on each of the four $T$-paths in Figure \ref{fig:T-paths}, so $f( \alpha,\beta ) = 1+1+1+1 = 4$ by the $T$-path formula.

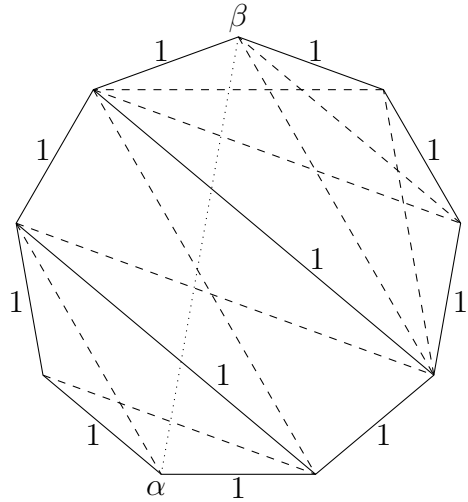
\begin{figure}
\begingroup
\[
  \begin{tikzpicture}[auto]
    \node[name=s, shape=regular polygon, regular polygon sides=9, minimum size=6cm, draw] {}; 

    \draw (90+2*360/9:3cm) to (90+5*360/9:3cm);
    \draw (90+1*360/9:3cm) to (90+6*360/9:3cm);

    \draw[dashed] (90+2*360/9:3cm) to (90+4*360/9:3cm);
    \draw[dashed] (90+3*360/9:3cm) to (90+5*360/9:3cm);

    \draw[dashed] (90+1*360/9:3cm) to (90+5*360/9:3cm);
    \draw[dashed] (90+2*360/9:3cm) to (90+6*360/9:3cm);    

    \draw[dashed] (90+0*360/9:3cm) to (90+6*360/9:3cm);
    \draw[dashed] (90+6*360/9:3cm) to (90+8*360/9:3cm);    
    \draw[dashed] (90+8*360/9:3cm) to (90+1*360/9:3cm);
    \draw[dashed] (90+1*360/9:3cm) to (90+7*360/9:3cm);    
    \draw[dashed] (90+7*360/9:3cm) to (90+0*360/9:3cm);

    \draw[dotted] (90+0*360/9:3cm) to (90+4*360/9:3cm);
    
    \draw (90+0.5*360/9:3cm) node {$1$};
    \draw (90+1.5*360/9:3cm) node {$1$};
    \draw (90+2.5*360/9:3cm) node {$1$};
    \draw (90+3.5*360/9:3cm) node {$1$};
    \draw (90+4.5*360/9:3cm) node {$1$};
    \draw (90+5.5*360/9:3cm) node {$1$};
    \draw (90+6.5*360/9:3cm) node {$1$};
    \draw (90+7.5*360/9:3cm) node {$1$};
    \draw (90+8.5*360/9:3cm) node {$1$};

    \draw (90+0*360/9:3.25cm) node {$\beta$};
    \draw (90+4*360/9:3.2cm) node {$\alpha$};

    \draw (1.05,0.05) node {$1$};
    \draw (-0.2,-1.5) node {$1$};
  \end{tikzpicture} 
\]
\endgroup
\caption{A weak frieze $f$ with respect to the dissection $D$ consisting of the two unbroken diagonals.  It has values in $\Rpos$ and is equal to $1$ on each edge, unbroken diagonal, and dashed diagonal (some of these values are shown).  It has been obtained by gluing together trivial weak friezes on the two squares and the pentagon created by $D$, see Theorem \ref{thm:B}.   The value $f( \alpha,\beta )$ on the dotted diagonal $\{ \alpha,\beta \}$ is $4$, see Section \ref{subsec:T-paths}.
}
\label{fig:9-gon}
\end{figure}

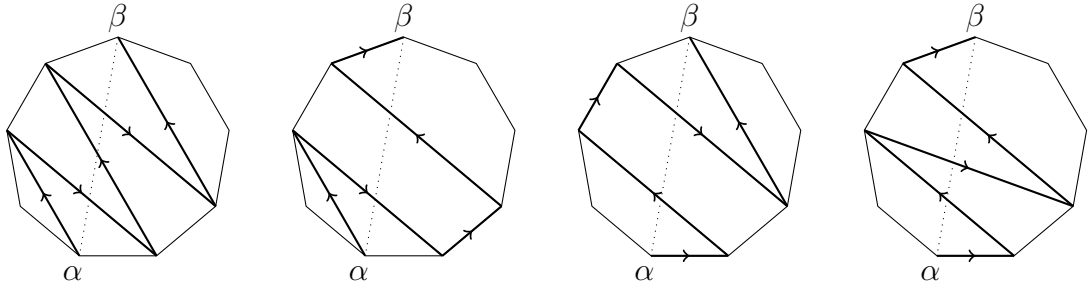
\begin{figure}
\begingroup
\[
  \begin{tikzpicture}[auto]

    \begin{scope}[shift={(-5.7,0)}]
      \node[name=s, shape=regular polygon, regular polygon sides=9, minimum size=3cm, draw] {}; 
      \draw (90+2*360/9:1.5cm) to (90+5*360/9:1.5cm);
      \draw (90+1*360/9:1.5cm) to (90+6*360/9:1.5cm);
      \draw[dotted] (90+0*360/9:1.5cm) to (90+4*360/9:1.5cm);    
      \draw (90+0*360/9:1.75cm) node {$\beta$};
      \draw (90+4*360/9:1.75cm) node {$\alpha$};
      \begin{scope}[thick,decoration={
    markings,mark=at position 0.5 with {\arrow{>}}}] 
        \draw[postaction={decorate}] (90+4*360/9:1.5cm)--(90+2*360/9:1.5cm);
        \draw[postaction={decorate}] (90+2*360/9:1.5cm)--(90+5*360/9:1.5cm);
        \draw[postaction={decorate}] (90+5*360/9:1.5cm)--(90+1*360/9:1.5cm);
        \draw[postaction={decorate}] (90+1*360/9:1.5cm)--(90+6*360/9:1.5cm);
        \draw[postaction={decorate}] (90+6*360/9:1.5cm)--(90+0*360/9:1.5cm);
      \end{scope}
    \end{scope}

    \begin{scope}[shift={(-1.9,0)}]
      \node[name=s, shape=regular polygon, regular polygon sides=9, minimum size=3cm, draw] {}; 
      \draw (90+2*360/9:1.5cm) to (90+5*360/9:1.5cm);
      \draw (90+1*360/9:1.5cm) to (90+6*360/9:1.5cm);
      \draw[dotted] (90+0*360/9:1.5cm) to (90+4*360/9:1.5cm);    
      \draw (90+0*360/9:1.75cm) node {$\beta$};
      \draw (90+4*360/9:1.75cm) node {$\alpha$};
      \begin{scope}[thick,decoration={
    markings,mark=at position 0.5 with {\arrow{>}}}] 
        \draw[postaction={decorate}] (90+4*360/9:1.5cm)--(90+2*360/9:1.5cm);
        \draw[postaction={decorate}] (90+2*360/9:1.5cm)--(90+5*360/9:1.5cm);
        \draw[postaction={decorate}] (90+5*360/9:1.5cm)--(90+6*360/9:1.5cm);
        \draw[postaction={decorate}] (90+6*360/9:1.5cm)--(90+1*360/9:1.5cm);
        \draw[postaction={decorate}] (90+1*360/9:1.5cm)--(90+0*360/9:1.5cm);
      \end{scope}
    \end{scope}

    \begin{scope}[shift={(1.9,0)}]
      \node[name=s, shape=regular polygon, regular polygon sides=9, minimum size=3cm, draw] {}; 
      \draw (90+2*360/9:1.5cm) to (90+5*360/9:1.5cm);
      \draw (90+1*360/9:1.5cm) to (90+6*360/9:1.5cm);
      \draw[dotted] (90+0*360/9:1.5cm) to (90+4*360/9:1.5cm);    
      \draw (90+0*360/9:1.75cm) node {$\beta$};
      \draw (90+4*360/9:1.75cm) node {$\alpha$};
      \begin{scope}[thick,decoration={
    markings,mark=at position 0.5 with {\arrow{>}}}] 
        \draw[postaction={decorate}] (90+4*360/9:1.5cm)--(90+5*360/9:1.5cm);
        \draw[postaction={decorate}] (90+5*360/9:1.5cm)--(90+2*360/9:1.5cm);
        \draw[postaction={decorate}] (90+2*360/9:1.5cm)--(90+1*360/9:1.5cm);
        \draw[postaction={decorate}] (90+1*360/9:1.5cm)--(90+6*360/9:1.5cm);
        \draw[postaction={decorate}] (90+6*360/9:1.5cm)--(90+0*360/9:1.5cm);
      \end{scope}
    \end{scope}

    \begin{scope}[shift={(5.7,0)}]
      \node[name=s, shape=regular polygon, regular polygon sides=9, minimum size=3cm, draw] {}; 
      \draw (90+2*360/9:1.5cm) to (90+5*360/9:1.5cm);
      \draw (90+1*360/9:1.5cm) to (90+6*360/9:1.5cm);
      \draw[dotted] (90+0*360/9:1.5cm) to (90+4*360/9:1.5cm);    
      \draw (90+0*360/9:1.75cm) node {$\beta$};
      \draw (90+4*360/9:1.75cm) node {$\alpha$};
      \begin{scope}[thick,decoration={
    markings,mark=at position 0.5 with {\arrow{>}}}] 
        \draw[postaction={decorate}] (90+4*360/9:1.5cm)--(90+5*360/9:1.5cm);
        \draw[postaction={decorate}] (90+5*360/9:1.5cm)--(90+2*360/9:1.5cm);
        \draw[postaction={decorate}] (90+2*360/9:1.5cm)--(90+6*360/9:1.5cm);
        \draw[postaction={decorate}] (90+6*360/9:1.5cm)--(90+1*360/9:1.5cm);
        \draw[postaction={decorate}] (90+1*360/9:1.5cm)--(90+0*360/9:1.5cm);
      \end{scope}
    \end{scope}

  \end{tikzpicture} 
\]
\endgroup
\caption{In Figure \ref{fig:9-gon}, these are the $T$-paths from $\alpha$ to $\beta$ with respect to the dissection $D$.}
\label{fig:T-paths}
\end{figure}

\smallskip
\subsection{Gluing}~\\
\label{subsec:gluing}
\vspace{-3.4ex}

The proof of Theorem \ref{thm:A} relies on our second main result.

\begin{ThmIntro}
[Gluing weak friezes]
\label{thm:B}
Let $P$ be a polygon, let $d_1, \ldots, d_m$ be pairwise non-crossing internal diagonals dividing $P$ into subpolygons $P_1, \ldots, P_{ m+1 }$, and let $D_i$ be a dissection of $P_i$ for each $i$.  The disjoint union $D = \{ d_1, \ldots, d_m \} \cupdot D_1 \cupdot \cdots \cupdot D_{ m+1 }$ is a dissection of $P$.

Let $f_i : \diag( P_i ) \xrightarrow{} K$ be a weak frieze with respect to $D_i$ for each $i$.  Assume that if $P_i$ and $P_j$ share a diagonal $d$, then $f_i( d ) = f_j( d )$ (note that such a $d$ must be an edge of both $P_i$ and $P_j$).

Then there is a unique weak frieze $f : \diag( P ) \xrightarrow{} K$ with respect to $D$ such that $f\Big|_{ \diag( P_i ) } = f_i$ for each $i$.  
\end{ThmIntro}

Our third main result concerns friezes.

\begin{ThmIntro}
[Gluing friezes]
\label{thm:C}
In Theorem \ref{thm:B}, the weak frieze $f$ is a frieze if and only if each $f_i$ is a frieze.  
\end{ThmIntro}

\smallskip
\subsection{Frieze patterns}~\\
\label{subsec:frieze_patterns}
\vspace{-3.4ex}

We end the introduction by explaining how (weak) friezes are related to (generalised) frieze patterns.

Let $n \geqslant 3$ be an integer and consider the coordinate system on a horizontal strip in Figure \ref{fig:coordinates}.  
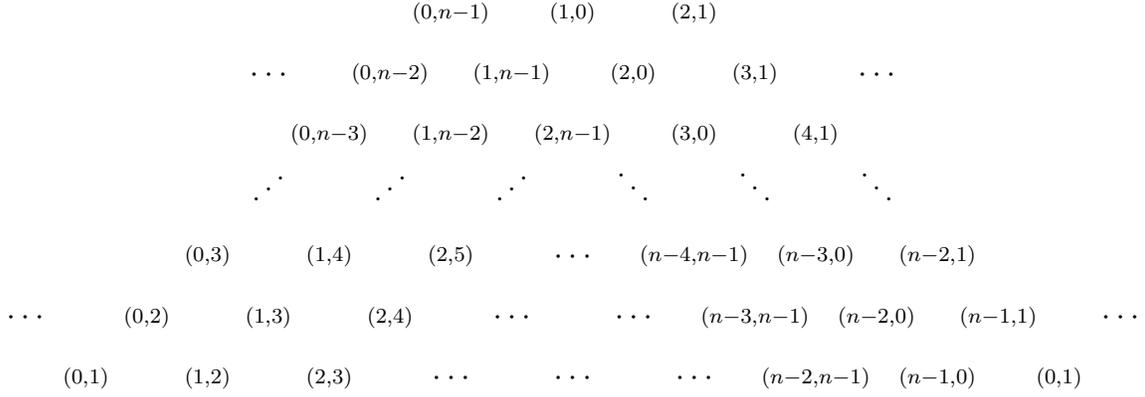
\begin{figure}
\begingroup
\[
  \xymatrix @-3.5pc @! {
    & & & & & & & *{\scriptstyle (0,n-1)} & & *{\scriptstyle (1,0)} & & *{\scriptstyle (2,1)} \\
    & & & & \cdots & & *{\scriptstyle (0,n-2)} & & *{\scriptstyle (1,n-1)} & & *{\scriptstyle (2,0)} & & *{\scriptstyle (3,1)} & & \cdots \\
    & & & & & *{\scriptstyle (0,n-3)} & & *{\scriptstyle (1,n-2)} & & *{\scriptstyle (2,n-1)} & & *{\scriptstyle (3,0)} & & *{\scriptstyle (4,1)} \\
    & & & & \adots & & \adots & & \adots & & \dddots & & \dddots & & \dddots \\
    & & & *{\scriptstyle (0,3)} & & *{\scriptstyle (1,4)} & & *{\scriptstyle (2,5)} & & \cdots & & *{\scriptstyle (n-4,n-1)} & & *{\scriptstyle (n-3,0)} & & *{\scriptstyle (n-2,1)} \\
    \cdots & & *{\scriptstyle (0,2)} & & *{\scriptstyle (1,3)} & & *{\scriptstyle (2,4)} & & \cdots & & \cdots & & *{\scriptstyle (n-3,n-1)} & & *{\scriptstyle (n-2,0)} & & *{\scriptstyle (n-1,1)} & & \cdots \\
    & *{\scriptstyle (0,1)} & & *{\scriptstyle (1,2)} & & *{\scriptstyle (2,3)} & & \cdots & & \cdots & & \cdots & & *{\scriptstyle (n-2,n-1)} & & *{\scriptstyle (n-1,0)} & & *{\scriptstyle (0,1)} \\
               }
\]
\endgroup
\caption{Coordinate system on a horizontal strip used for turning (weak) friezes on the $n$-gon into (generalised) frieze patterns.}
\label{fig:coordinates}
\end{figure}
The coordinates are elements of $\BZ/n$, and taking a step right adds $1$ to each coordinate.  The first coordinate is constant when ascending diagonally, then second when descending diagonally.  A (weak) frieze $f$ on an $n$-gon $P$ is turned into a pattern on the strip by labelling the vertices of $P$ by $\BZ/n$ and placing the value $f( i,j )$ at position $( i,j )$ in the coordinate system.  Note that since $f( i,j ) = f( j,i )$, the pattern has a glide symmetry.

First, it is classic that if $f$ is a frieze with values in $\Rpos$ satisfying $f( i,i+1 ) = 1$ for each $i \in \BZ/n$, then the corresponding pattern is a frieze pattern as defined in \cite[sec.\ 1]{Cox}.  Namely, if $i,j \in \BZ/n$ are non-neighbouring vertices, then $\{ i,j \}$ and $\{ i+1,j+1 \}$ are crossing diagonals and Equation \eqref{equ:Ptolemy} implies
\[
  f( i,j )f( i+1,j+1 ) - f( i,j+1 )f( i+1,j ) = 1.
\]
This is the unimodular equation from  \cite[sec.\ 1]{Cox}, see Figure \ref{fig:diamond}.  Every frieze pattern in the sense of \cite[sec.\ 1]{Cox} arises like this by \cite[thm.\ 2.5]{HJ}.
\begin{figure}
\begingroup
\[
  \xymatrix @! @R-6.5pc @C-4.0pc {
    & f( i,j+1 ) & \\
    f( i,j ) \!\!\!\!\!\!\!\!\!\!\!\!\!\!\!\! & & f( i+1,j+1 ) \\
    & f( i+1,j ) & \\
               }
\]
\endgroup
\caption{If $f$ is a frieze with values in $\Rpos$ satisfying $f( i,i+1 ) = 1$, then $f$ satisfies $f( i,j )f( i+1,j+1 ) - f( i,j+1 )f( i+1,j ) = 1$, the unimodular equation.  The four values of $f$ in the equation constitute a ``diamond'' in the corresponding frieze pattern as shown.}
\label{fig:diamond}
\end{figure}

Secondly, if $f$ is any frieze with values in $\Rpos$, then the corresponding pattern is a frieze pattern with coefficients as defined in \cite[sec.\ 3]{Propp}, see \cite[def.\ 2.1]{CHJ}.  Equation \eqref{equ:Ptolemy} becomes the equation in \cite[def.\ 2.1(iii)]{CHJ}.  Every frieze pattern with coefficients which has positive entries arises like this by \cite[thm.\ 3.3]{CHJ}.

Thirdly, if $f$ is a weak frieze with values in $\BN = \{ 1,2,\ldots \}$ satisfying $f( i,i+1 ) = 1$ for each $i \in \BZ/n$, then the corresponding pattern is a generalised frieze pattern as defined in \cite[sec.\ 5]{BHJ}.  For instance, the weak frieze $f$ in Figure \ref{fig:9-gon} gives the generalised frieze pattern in Figure \ref{fig:9-gon_pattern}.  The entries in the first and last row are $1$ because they are the values $f(i,i+1)$.  The entry $f( \alpha,\beta ) = 4$ is shown in red.
\begin{figure}
\begingroup
\[
  \xymatrix @-2.5pc @! {
    & 1 & & 1 & & 1 & & 1 & & 1 & & 1 & & 1 & & 1 & & 1 & & 1 & & 1 \\
    & & 2 & & 2 & & 1 & & 1 & & 1 & & 2 & & 2 & & 1 & & 1 & & 2 \\
    \cdots & 2 & & 4 & & 2 & & 1 & & 1 & & 2 & & *+[red]{4} & & 2 & & 1 & & 2 & & 4 & \cdots \\
    & & 4 & & 4 & & 2 & & 1 & & 2 & & 4 & & 4 & & 1 & & 1 & & 4 \\
    & 2 & & 4 & & 4 & & 1 & & 1 & & 4 & & 4 & & 2 & & 1 & & 2 & & 4 \\
    \cdots & & 2 & & *+[red]{4} & & 2 & & 1 & & 2 & & 4 & & 2 & & 1 & & 1 & & 2 & & \cdots \\
    & 1 & & 
    2 & & 2 & & 1 & & 1 & & 2 & & 2 & & 1 & & 1 & & 1 & & 2 \\
    & & 
    1 & & 
    1 & & 1 & & 1 & & 1 & & 1 & & 1 & & 1 & & 1 & & 1 \\
               }
\]
\endgroup
\caption{The generalised frieze pattern corresponding to the weak frieze in Figure \ref{fig:9-gon}.}
\label{fig:9-gon_pattern}
\end{figure}

The paper is organised as follows: Section \ref{sec:definitions} collects some definitions, Section \ref{sec:Thm_BC} proves Theorems \ref{thm:B} and \ref{thm:C}, Section \ref{sec:lemmas_on_T-paths} establishes some properties of $T$-paths, and Section \ref{sec:Thm_A} proves Theorem \ref{thm:A}.

\section{Semifields and polygons}
\label{sec:definitions}

Recall that throughout the paper, $( K,+,\cdot )$ is a fixed semifield in the following sense.

\begin{Definition}
\label{def:semifield}
A {\it semifield} is a triple $( K,+,\cdot )$, where $K$ is a set and $+$ and $\:\cdot\:$ are binary operations, satisfying the following.
\begin{itemize}
\setlength\itemsep{4pt}

  \item  The operation $+$ is associative and commutative.

  \item  The operation $\:\cdot\:$ turns $K$ into a commutative group.  The unit element is denoted $1_K$, the inverse of $x$ by $x^{-1}$. 

  \item  The operation $\:\cdot\:$ distributes over $+$.

\end{itemize}
We abbreviate $( K,+,\cdot )$ to $K$ and write $xy := x \cdot y$ and $\frac{x}{y} := x \cdot y^{-1}$.
\end{Definition}

\begin{Definition}
\label{def:polygon}
A {\em polygon} $P$ is a finite set $V$ of three or more {\em vertices} with a cyclic order.

The predecessor and successor of $\alpha \in V$ are denoted $\alpha^-$ and $\alpha^+$.

A {\em diagonal} of $P$ is a two-element subset  of $V$.  The set of diagonals is denoted $\diag(P)$.

The diagonal $\{ \alpha,\beta \}$ has {\em end points} $\alpha$ and $\beta$.  Diagonals of the form $\{ \alpha,\alpha^+ \}$ are called {\em edges}.  The remaining diagonals are called {\em internal diagonals}.

The diagonals $\{ \alpha,\beta \}$ and $\{ \gamma,\delta \}$ {\em cross} if $\alpha,\beta,\gamma,\delta$ are four distinct vertices which appear in the order $\alpha,\gamma,\beta,\delta$ or $\alpha,\delta,\beta,\gamma$ in $V$.  

A {\em subpolygon} is a subset of $V$ of three or more vertices equipped with the induced cyclic order.

A {\em dissection} of $P$ is a set $D$ of pairwise non-crossing internal diagonals.  If $D$ has $m$ elements, then it divides $P$ into $m+1$ subpolygons.  Note that $D$ can be empty.

If $V$ has $n$ elements, then $P$ is called an {\em $n$-gon}, and we often imagine it realised as a convex $n$-angle in the Euclidean plane.  This means that we can make sense of
Definition \ref{def:T-paths_and_T-path_formula}(iii), which says that the diagonals $\{ \pi_{ 2j },\pi_{ 2j+1 } \}$ cross the diagonal $\{ \pi_1,\pi_p \}$ ``at pairwise different points which progress monotonically in the direction from $\pi_1$ to $\pi_p$''.  In our figures of $n$-angles, the positive direction is anticlockwise.
\end{Definition}

\begin{Definition}
[$T$-paths II]
\label{def:T-paths_2}
We extend the notation of Definition \ref{def:T-paths_and_T-path_formula} as follows:  If $\rho_1, \ldots, \rho_t$ are vertices of $P$, then we set
\[
  \TPath{P}{D}{\pi_1}{\pi_p}^{ (\rho_1,\ldots,\rho_t ) }
  =
  \{\, \pi \in \TPath{P}{D}{\pi_1}{\pi_p} \,|\, \pi_1 = \rho_1, \ldots, \pi_t = \rho_t \,\}.
\]
The notation is further extended by permitting the sign $\neq$, for instance,
\[
  \TPath{P}{D}{\pi_1}{\pi_p}^{ (\rho_1,\neq \rho_2 ) }
  =
  \{\, \pi \in \TPath{P}{D}{\pi_1}{\pi_p} \,|\, \pi_1 = \rho_1, \pi_2 \neq \rho_2 \,\}.
\]
\end{Definition}

\section{Proofs of Theorems \ref{thm:B} and \ref{thm:C}}
\label{sec:Thm_BC}

In the proofs of Theorems \ref{thm:B} and \ref{thm:C} we will assume $m = 1$, that is, there is one diagonal $d_1$ dividing $P$ into subpolygons $P_1$ and $P_2$.  This implies the general case by an easy induction.  Note that the proofs do not use subtraction, which is unavailable in the semifield $K$.

Let $P$ have the set of vertices $V$, denote $d_1$ by $d$, and pick $\zeta,\eta \in V$ such that $d = \{ \zeta,\eta \}$ while the sets of vertices of $P_1$ and $P_2$ are $V_1 = \{\, \varepsilon \in V \,|\, \zeta \leqslant \varepsilon \leqslant \eta \,\}$ and $V_2 = \{\, \varepsilon \in V \,|\, \eta \leqslant \varepsilon \leqslant \zeta \,\}$, see Figure \ref{fig:ThmB1}.  Set 
\[
  U_1 = \{\, \varepsilon \in V \,|\, \zeta < \varepsilon < \eta \,\}
  \;\;,\;\;
  U_2 = \{\, \varepsilon \in V \,|\, \eta < \varepsilon < \zeta \,\}.
\]
There are disjoint unions $D = D_1 \cupdot \{ d \} \cupdot D_2 $ and
\[
  V = U_1 \cupdot \{ \zeta,\eta \} \cupdot U_2.
\]

\begin{proof}[Proof of Theorem \ref{thm:B}]
Since $d = \{ \zeta,\eta \}$ is in $D$, by Definition \ref{def:friezes_and_weak_friezes}(ii) a weak frieze $f$ with respect to $D$ must satisfy
\[
  f( \zeta,\eta )f( \alpha,\beta )
  =
  f( \zeta,\alpha )f( \eta,\beta )
  +
  f( \zeta,\beta )f( \eta,\alpha )
\]
when $\{ \zeta,\eta \}$ and $\{ \alpha,\beta \}$ are crossing diagonals of $P$.  If $f$ also satisfies $f\Big|_{ \diag( P_i ) } = f_i$ for $i \in \{ 1,2 \}$ and $x := f_1( \zeta,\eta ) = f_2( \zeta,\eta )$, then $f( \alpha,\beta )$ must be given by the entries of the following table, according to whether $\alpha$ and $\beta$ are in $U_1$, $U_2$, or $\{ \zeta,\eta \}$.
\label{page:with_table}
\[
\mbox{
\begin{tabular}{c|cccccc}
  \diaghead{\theadfont xxxxxxxxx}{$\alpha$}{$\beta$} & $\scriptstyle U_1$ & $\scriptstyle \{ \zeta,\eta \}$ & $\scriptstyle U_2$ \\[2mm] \cline{1-4}
  $\scriptstyle U_1$ & $\scriptstyle f_1( \alpha,\beta )$ & $\scriptstyle f_1( \alpha,\beta )$ & $\scriptstyle x^{-1}[ f_1( \zeta,\alpha )f_2( \eta,\beta ) + f_2( \zeta,\beta )f_1( \eta,\alpha ) ]$ \\[2mm]
  $\scriptstyle \{ \zeta,\eta \}$ & $\scriptstyle f_1( \alpha,\beta )$ & $\scriptstyle f_1( \alpha,\beta ) = f_2( \alpha,\beta )$ & $\scriptstyle f_2( \alpha,\beta )$ \\[2mm]
  $\scriptstyle U_2$ & $\scriptstyle x^{-1}[ f_1( \zeta,\beta )f_2( \eta,\alpha ) + f_2( \zeta,\alpha )f_1( \eta,\beta ) ]$ & $\scriptstyle f_2( \alpha,\beta )$ & $\scriptstyle f_2( \alpha,\beta )$ \\
\end{tabular}
}
\]
This shows uniqueness of the weak frieze $f$ claimed in Theorem \ref{thm:B}.  

To show existence of $f$, let $f( \alpha,\beta )$ be defined by the table.  We must show that $f$ is a weak frieze with respect to $D$; that is, for crossing diagonals $\{ \alpha,\beta \}$ and $\{ \gamma,\delta \}$ with $\{ \gamma,\delta \} \in D$, Equation \eqref{equ:Ptolemy} holds.  It is necessary to treat a number of cases.  We leave most of them to the reader, but show the computation for the case
\[
  \zeta < \alpha < \eta < \gamma < \beta < \delta < \zeta,
\]
see Figure \ref{fig:ThmB1} (left), in which
\begin{align}
\label{equ:ThmB3}
  \alpha \in U_1, \\
\label{equ:ThmB4}
  \beta,\gamma,\delta \in U_2.
\end{align}

\begin{figure}
\begingroup
\[
  \begin{tikzpicture}[auto]
    \begin{scope}[shift={(-4.2,0)}]
      \node[name=s, shape=regular polygon, regular polygon sides=20, minimum size=6cm, draw] {}; 
      \draw (8.5*360/20:3.25cm) node {$\zeta$};
      \draw (11.5*360/20:3.25cm) node {$\alpha$};
      \draw (15.5*360/20:3.25cm) node {$\eta$};
      \draw (19.5*360/20:3.25cm) node {$\gamma$};
      \draw (21.5*360/20:3.25cm) node {$\beta$};
      \draw (23.5*360/20:3.25cm) node {$\delta$};
      \draw (8.5*360/20:3cm) to (15.5*360/20:3cm);
      \draw (-1.7,-1.3) node {$P_1$};
      \draw (0.8,0.8) node {$P_2$};
    \end{scope}
    
    \begin{scope}[shift={(4.2,0)}]
      \node[name=s, shape=regular polygon, regular polygon sides=20, minimum size=6cm, draw] {}; 
      \draw (8.5*360/20:3.25cm) node {$\zeta$};
      \draw (10.5*360/20:3.25cm) node {$\alpha$};
      \draw (13.5*360/20:3.25cm) node {$\gamma$};
      \draw (15.5*360/20:3.25cm) node {$\eta$};
      \draw (20.5*360/20:3.25cm) node {$\beta$};
      \draw (23.5*360/20:3.25cm) node {$\delta$};
      \draw (8.5*360/20:3cm) to (15.5*360/20:3cm);
      \draw (-1.7,-1.3) node {$P_1$};
      \draw (0.8,0.8) node {$P_2$};  
    \end{scope}
    
  \end{tikzpicture} 
\]
\endgroup
\caption{The configurations in the proofs of Theorems \ref{thm:B} and \ref{thm:C}.}
\label{fig:ThmB1}
\end{figure}
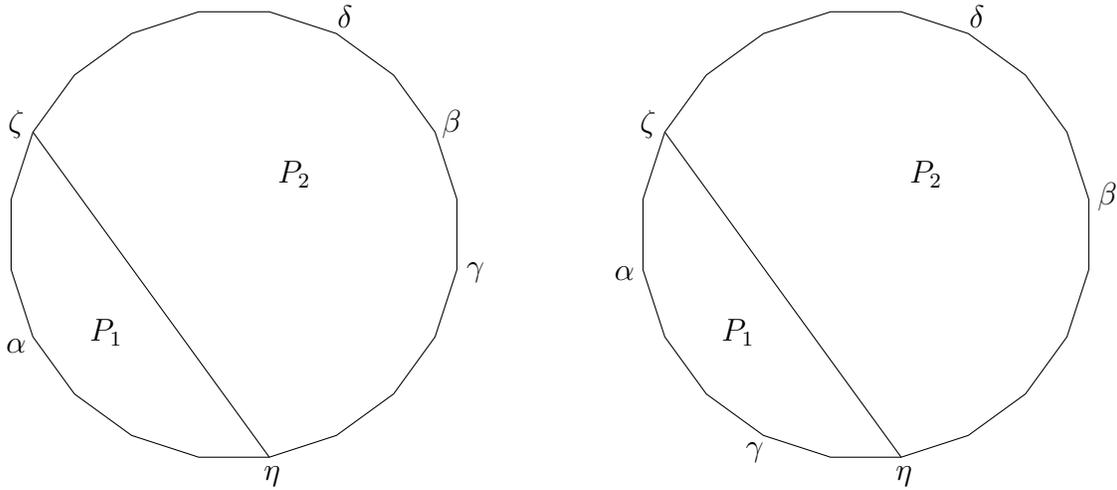

Equation \eqref{equ:ThmB4} implies that $\{ \gamma,\delta \}$, $\{ \zeta,\beta \}$, $\{ \eta,\beta \}$ are diagonals of $P_2$, so in particular $\{ \gamma,\delta \} \in D_2$.  Since $\{ \gamma,\delta \}$ crosses each of $\{ \zeta,\beta \}$ and $\{ \eta,\beta \}$ while $f_2$ is a weak frieze with respect to $D_2$, we get the following equations.
\begin{align}
\label{equ:ThmB1}
  f_2( \zeta,\beta )f_2( \gamma,\delta )
  & = f_2( \zeta,\gamma )f_2( \beta,\delta )
    + f_2( \zeta,\delta )f_2( \beta,\gamma ) \\
\label{equ:ThmB2}
  f_2( \eta,\beta )f_2( \gamma,\delta )
  & = f_2( \eta,\gamma )f_2( \beta,\delta )
    + f_2( \eta,\delta )f_2( \beta,\gamma )
\end{align}
The definition of $f$ and Equations \eqref{equ:ThmB3} and \eqref{equ:ThmB4} give
\begin{equation}
\label{equ:ThmB5}
  f( \alpha,\beta )
  = x^{-1}
    \big[ f_1( \zeta,\alpha )f_2( \eta,\beta )
    + f_2( \zeta,\beta )f_1( \eta,\alpha ) \big]
\end{equation}
and analogous equations where $\beta$ is replaced with $\gamma$ or $\delta$.
Hence
\begin{align*}
  \begingroup \color{red}
  f( \alpha,\beta )
  \endgroup f( \gamma,\delta )
  & \stackrel{\rm (a)}{=} x^{-1}
    \big[ f_1( \zeta,\alpha )f_2( \eta,\beta )
    + f_2( \zeta,\beta )f_1( \eta,\alpha ) \big] f_2( \gamma,\delta ) \\[1.75mm]
  & = x^{-1} f_1( \zeta,\alpha )
  \begingroup \color{red}
  f_2( \eta,\beta )f_2( \gamma,\delta )
  \endgroup \\
  & \;\;\;\; + x^{-1} f_1( \eta,\alpha )
  \begingroup \color{red}
  f_2( \zeta,\beta )f_2( \gamma,\delta ) 
  \endgroup \\[1.75mm]
  & \stackrel{\rm (b)}{=} x^{-1} f_1( \zeta,\alpha ) \big[ f_2( \eta,\gamma )f_2( \beta,\delta ) + f_2( \eta,\delta )f_2( \beta,\gamma ) \big] \\
  & \;\;\;\; +
    x^{-1} f_1( \eta,\alpha ) \big[ f_2( \zeta,\gamma )f_2( \beta,\delta ) + f_2( \zeta,\delta )f_2( \beta,\gamma ) \big] \\[1.75mm]
  & = 
  \begingroup \color{red}
  x^{-1} \big[ f_1( \zeta,\alpha )f_2( \eta,\delta ) + f_2( \zeta,\delta )f_1( \eta,\alpha ) \big]
  \endgroup f_2( \beta,\gamma ) \\
  & \;\;\;\; + 
  \begingroup \color{red}
  x^{-1} \big[ f_1( \zeta,\alpha )f_2( \eta,\gamma ) + f_2( \zeta,\gamma ) f_1( \eta,\alpha ) \big] 
  \endgroup f_2( \beta,\delta ) \\[1.75mm]
  & \stackrel{\rm (c)}{=} f( \alpha,\delta )f( \beta,\gamma ) + f( \alpha,\gamma )f( \beta,\delta ) 
\end{align*}
as desired, where each red factor is replaced in the subsequent step.  For (a) and (c), use Equation \eqref{equ:ThmB5} and analogous equations where $\beta$ is replaced with $\gamma$ or $\delta$.  For (b), use Equations \eqref{equ:ThmB1} and \eqref{equ:ThmB2}.  The remaining equalities are simple computations.
\end{proof}

\begin{proof}[Proof of Theorem \ref{thm:C}]
``Only if'' is clear.

``If'':  Assume that $f_1$ and $f_2$ are friezes.  We must show that $f$ is a frieze; that is, for crossing diagonals $\{ \alpha,\beta \}$ and $\{ \gamma,\delta \}$, Equation \eqref{equ:Ptolemy} holds.  It is necessary to treat a number of cases.  We leave most of them to the reader, but show the computation for the case
\[
  \zeta < \alpha < \gamma < \eta < \beta < \delta < \zeta,
\]
see Figure \ref{fig:ThmB1} (right), in which
\begin{align}
\label{equ:ThmC1}
  \alpha,\gamma \in U_1, \\
\label{equ:ThmC2}
  \beta,\delta \in U_2.
\end{align}

Equation \eqref{equ:ThmC1} implies that $\{ \zeta,\gamma \}$, $\{ \eta,\alpha \}$ are diagonals of $P_1$.  Since they cross while $f_1$ is a frieze, we get the first of the following equations, and the second follows by an analogous argument.
\begin{align}
\label{equ:ThmC4}
  f_1( \zeta,\gamma )f_1( \eta,\alpha )
  & = f_1( \zeta,\eta )f_1( \gamma,\alpha )
    + f_1( \zeta,\alpha )f_1( \gamma,\eta ) \\
\label{equ:ThmC5}
  f_2( \zeta,\beta )f_2( \eta,\delta )
  & = f_2( \zeta,\eta )f_2( \beta,\delta )
    + f_2( \zeta,\delta )f_2( \beta,\eta )
\end{align}
Equations \eqref{equ:ThmC1} and \eqref{equ:ThmC2} mean that Equation \eqref{equ:ThmB5} is still valid, and so are the analogous equations where the pair $( \alpha,\beta )$ is replaced by one of $( \alpha,\delta )$, $( \gamma,\beta )$, $( \gamma,\delta )$.  Hence
\begin{align*}
  & 
  \begingroup \color{red}
  x^2f( \alpha,\beta )f( \gamma,\delta ) 
  \endgroup
  \\[1.75mm]
  & \;\; \stackrel{\rm (a)}{=} \big[ f_1( \zeta,\alpha )f_2( \eta,\beta ) + f_2( \zeta,\beta )f_1( \eta,\alpha ) \big] \big[ f_1( \zeta,\gamma )f_2( \eta,\delta ) + f_2( \zeta,\delta )f_1( \eta,\gamma ) \big] \\[1.75mm]
  & \;\; = f_1( \zeta,\alpha )f_2( \eta,\beta )f_1( \zeta,\gamma )f_2( \eta,\delta )
  + f_1( \zeta,\alpha )f_2( \eta,\beta )f_2( \zeta,\delta )f_1( \eta,\gamma ) \\
  & \;\;\;\;\;\;
  + \begingroup \color{red}
  f_2( \zeta,\beta )
  \endgroup
  f_1( \eta,\alpha )f_1( \zeta,\gamma )
  \begingroup \color{red}
  f_2( \eta,\delta )
  \endgroup
  + f_2( \zeta,\beta )f_1( \eta,\alpha )f_2( \zeta,\delta )f_1( \eta,\gamma ) \\[1.75mm]
  & \;\; \stackrel{\rm (b)}{=} f_1( \zeta,\alpha )f_2( \eta,\beta )f_1( \zeta,\gamma )f_2( \eta,\delta )
  + f_1( \zeta,\alpha )f_2( \eta,\beta )f_2( \zeta,\delta )f_1( \eta,\gamma ) \\
  & \;\;\;\;\;\; + \big[ f_2( \zeta,\eta )f_2( \beta,\delta ) + f_2( \zeta,\delta )f_2( \beta,\eta ) \big]f_1( \eta,\alpha )f_1( \zeta,\gamma ) \\
  & \;\;\;\;\;\; + f_2( \zeta,\beta )f_1( \eta,\alpha )f_2( \zeta,\delta )f_1( \eta,\gamma ) \\[1.75mm]
  & \;\; = f_1( \zeta,\alpha )f_2( \eta,\beta )f_1( \zeta,\gamma )f_2( \eta,\delta )
  + f_1( \zeta,\alpha )f_2( \eta,\beta )f_2( \zeta,\delta )f_1( \eta,\gamma ) \\
  & \;\;\;\;\;\; + f_2( \zeta,\eta )f_2( \beta,\delta )
  \begingroup \color{red}
  f_1( \eta,\alpha )f_1( \zeta,\gamma )
  \endgroup + f_2( \zeta,\delta )f_2( \beta,\eta )f_1( \eta,\alpha )f_1( \zeta,\gamma ) \\
  & \;\;\;\;\;\; + f_2( \zeta,\beta )f_1( \eta,\alpha )f_2( \zeta,\delta )f_1( \eta,\gamma ) \\[1.75mm]
  & \;\; \stackrel{\rm (c)}{=} f_1( \zeta,\alpha )f_2( \eta,\beta )f_1( \zeta,\gamma )f_2( \eta,\delta )
  + f_1( \zeta,\alpha )f_2( \eta,\beta )f_2( \zeta,\delta )f_1( \eta,\gamma ) \\
  & \;\;\;\;\;\; + f_2( \zeta,\eta )f_2( \beta,\delta ) \big[ f_1( \zeta,\eta )f_1( \gamma,\alpha ) + f_1( \zeta,\alpha )f_1( \gamma,\eta ) \big] \\
  & \;\;\;\;\;\; + f_2( \zeta,\delta )f_2( \beta,\eta )f_1( \eta,\alpha )f_1( \zeta,\gamma ) + f_2( \zeta,\beta )f_1( \eta,\alpha )f_2( \zeta,\delta )f_1( \eta,\gamma ) \\[1.75mm]
  & \;\; = f_1( \zeta,\alpha )f_2( \eta,\beta )f_1( \zeta,\gamma )f_2( \eta,\delta ) + f_2( \zeta,\eta )f_2( \beta,\delta )f_1( \zeta,\eta )f_1( \gamma,\alpha ) \\
  & \;\;\;\;\;\; + f_2( \zeta,\delta )f_2( \beta,\eta )f_1( \eta,\alpha )f_1( \zeta,\gamma ) + f_2( \zeta,\beta )f_1( \eta,\alpha )f_2( \zeta,\delta )f_1( \eta,\gamma ) \\
  & \;\;\;\;\;\; + f_1( \zeta,\alpha )f_1( \eta,\gamma )
\begingroup \color{red}
\big[ f_2( \eta,\beta )f_2( \zeta,\delta ) + f_2( \zeta,\eta )f_2( \beta,\delta ) \big]
\endgroup
  \\[1.75mm]
  & \;\; \stackrel{\rm (d)}{=} f_1( \zeta,\alpha )f_2( \eta,\beta )f_1( \zeta,\gamma )f_2( \eta,\delta ) + f_2( \zeta,\eta )f_2( \beta,\delta )f_1( \zeta,\eta )f_1( \gamma,\alpha ) \\
  & \;\;\;\;\;\; + f_2( \zeta,\delta )f_2( \beta,\eta )f_1( \eta,\alpha )f_1( \zeta,\gamma ) + f_2( \zeta,\beta )f_1( \eta,\alpha )f_2( \zeta,\delta )f_1( \eta,\gamma ) \\
  & \;\;\;\;\;\; + f_1( \zeta,\alpha )f_1( \eta,\gamma )f_2( \zeta,\beta )f_2( \eta,\delta )
  \\[1.75mm]
  & \;\; = x^2 f( \gamma,\alpha )f_2( \beta,\delta ) \\
  & \;\;\;\;\;\; + 
  \begingroup \color{red}
  \big[ f_1( \zeta,\alpha )f_2( \eta,\delta ) + f_2( \zeta,\delta )f_1( \eta,\alpha ) \big]
  \big[ f_1( \zeta,\gamma )f_2( \eta,\beta ) + f_2( \zeta,\beta )f_1( \eta,\gamma ) \big] 
  \endgroup
  \\[1.75mm]
  & \;\; \stackrel{\rm (e)}{=} x^2 f( \alpha,\gamma )f( \beta,\delta ) + x^2 f( \alpha,\delta )f( \beta,\gamma ).
\end{align*}
Multiplying by $x^{-2}$ gives Equation \eqref{equ:Ptolemy} as desired.  Each red factor is replaced in the subsequent step. To get (a) and (e), use Equation \eqref{equ:ThmB5} and analogous equations where the pair $( \alpha,\beta )$ is replaced by one of $( \alpha,\delta )$, $( \gamma,\beta )$, $( \gamma,\delta )$.  To get (b) and (d), use Equation \eqref{equ:ThmC5}.  To get (c), use Equation \eqref{equ:ThmC4}.  The remaining equalities are simple computations.
\end{proof}

\section{Lemmas on $T$-paths}
\label{sec:lemmas_on_T-paths}

Recall that $T$-paths, the notation $\TPath{P}{D}{\alpha}{\beta}$, and the $T$-path formula were introduced in De\-fi\-ni\-ti\-on \ref{def:T-paths_and_T-path_formula}.  Additional notation was introduced in Definition \ref{def:T-paths_2}, and we will use this material without further comment.

\begin{Lemma}
\label{lem:inversion}
Let $P$ be a polygon with a dissection $D$ and let $\alpha \neq \beta$ be vertices of $P$.
\begin{enumerate}
\setlength\itemsep{4pt}

  \item  There is a bijection $\TPath{P}{D}{\alpha}{\beta} \xrightarrow{} \TPath{P}{D}{\beta}{\alpha}$ given by $( \pi_1, \ldots, \pi_p ) \mapsto ( \pi_p, \ldots, \pi_1 )$.
  
  \item  If $f : \diag(P) \xrightarrow{} K$ is a map, then
\[
  \sum_{ \pi \in \TPath{P}{D}{\alpha}{\beta} } f( \pi )
  =
  \sum_{ \rho \in \TPath{P}{D}{\beta}{\alpha} } f( \rho ).
\]

\end{enumerate}
\end{Lemma}

\begin{proof}
If $( \pi_1, \ldots, \pi_p ) \in \TPath{P}{D}{\alpha}{\beta}$, then $\{ \pi_{ p-1 },\pi_p \}$ shares an end point with $\{ \alpha,\beta \} = \{ \pi_1,\pi_p \}$.  It follows that these diagonals do not cross, so $p-1$ is odd by Definition \ref{def:T-paths_and_T-path_formula}(iii).  Hence $p$ is even, which makes it easy to check part (i), and implies that Equation \eqref{equ:f_on_T-paths} gives $f\big( ( \pi_1, \ldots, \pi_p ) \big) = f\big( ( \pi_p, \ldots, \pi_1 ) \big)$.  Combining this equation with part (i) proves part (ii). 
\end{proof}

\begin{Lemma}
\label{lem:trivial_T-paths}
Let $P$ be a polygon with a dissection $D$ and assume that the diagonal $\{ \alpha,\beta \}$ crosses no diagonal in $D$.  Then:
\begin{enumerate}
\setlength\itemsep{4pt}

  \item  $\TPath{P}{D}{\alpha}{\beta} = \{ ( \alpha,\beta ) \}$.

  \item  If $f : \diag(P) \xrightarrow{} K$ is a map, then Equation \eqref{equ:the_T-path_formula} holds.

\end{enumerate}
\end{Lemma}

\begin{proof}
(i):  If $\pi = ( \pi_1, \ldots, \pi_p ) \in \TPath{P}{D}{\alpha}{\beta}$ then by Definition \ref{def:T-paths_and_T-path_formula}(iii) there is no diagonal which could appear as $\{ \pi_2,\pi_3 \}$, since $\{ \alpha,\beta \}$ crosses no diagonal in $D$.    Hence $p = 2$ and the only possibility is $\pi = ( \alpha,\beta )$, which is indeed in $\TPath{P}{D}{\alpha}{\beta}$.  

(ii):  Equation \eqref{equ:the_T-path_formula} holds because the sum on the right hand side has only the term $f( \alpha,\beta )$ by part (i) of the lemma.
\end{proof}

\begin{Setup}
[Ears]
\label{set:ear}
Let $P$ be a polygon, $V$ the set of vertices of $P$, and let $D$ be a non-empty dissection of $P$.  It is easy to see that we can accomplish the following setup, which is sketched in Figure \ref{fig:P1P2} and will be assumed for the rest of Section \ref{sec:lemmas_on_T-paths}:
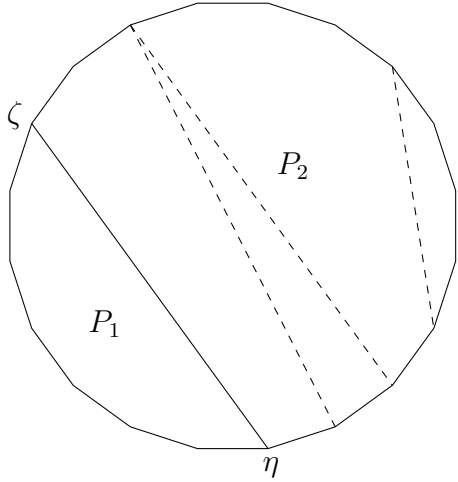
\begin{figure}
\begingroup
\[
  \begin{tikzpicture}[auto]
    \node[name=s, shape=regular polygon, regular polygon sides=20, minimum size=6cm, draw] {}; 
    \draw (8.5*360/20:3.25cm) node {$\zeta$};
    \draw (15.5*360/20:3.25cm) node {$\eta$};
    \draw (8.5*360/20:3cm) to (15.5*360/20:3cm);
    \draw [dashed] (6.5*360/20:3cm) to (16.5*360/20:3cm);
    \draw [dashed] (6.5*360/20:3cm) to (17.5*360/20:3cm);
    \draw [dashed] (2.5*360/20:3cm) to (18.5*360/20:3cm);
    \draw (-1.7,-1.3) node {$P_1$};
    \draw (0.8,0.8) node {$P_2$};
  \end{tikzpicture} 
\]
\endgroup
\caption{The polygon $P$ is divided into subpolygons $P_1$ and $P_2$ by the diagonal $d = \{ \zeta,\eta \}$.  The dissection $D$ of $P$ has the form $D = \{ d \} \cupdot D_2$, where $D_2$ (dashed) is a dissection of $P_2$.}
\label{fig:P1P2}
\end{figure}

The diagonal $d = \{ \zeta,\eta \}$ in $D$ divides $P$ into subpolygons $P_1$ and $P_2$ such that $D = \{ d \} \cupdot D_2$ for a dissection $D_2$ of $P_2$.  The sets of vertices of $P_1$ and $P_2$ are
\[
  V_1 = \{\, \varepsilon \in V \,|\, \zeta \leqslant \varepsilon \leqslant \eta \,\}
  \;\;,\;\;
  V_2 = \{\, \varepsilon \in V \,|\, \eta \leqslant \varepsilon \leqslant \zeta \,\}, 
\]
and we set
\[
  U_1 = \{\, \varepsilon \in V \,|\, \zeta < \varepsilon < \eta \,\}
  \;\;,\;\;
  U_2 = \{\, \varepsilon \in V \,|\, \eta < \varepsilon < \zeta \,\}.
\]
There are disjoint unions
\[
  V
  \;=\;
  V_1 \cupdot U_2
  \;=\;
  U_1 \cupdot V_2
  \;=\;
  U_1 \cupdot \{ \zeta,\eta \} \cupdot U_2. 
\]
Observe that no diagonal in $D$ has an end point in $U_1$.  In particular, $D$ contains no internal diagonals of $P_1$.  Accordingly, we say that $P_1$ is an {\em ear}.  
\end{Setup}

\begin{Lemma}
\label{lem:14}
Assume that $\alpha,\beta \in V_2$.  Then $\TPath{P}{D}{\alpha}{\beta} = \TPath{P_2}{D_2}{\alpha}{\beta}$.
\end{Lemma}

\begin{proof}
It is easy to show that the inclusion $\supseteq$ holds, and that for the inclusion $\subseteq$ we only need to show that each $\pi = ( \pi_1, \ldots, \pi_p ) \in \TPath{P}{D}{\alpha}{\beta}$ satisfies $\pi_i \in V_2$ for each $i$.  Assume the opposite and pick $\pi = ( \pi_1, \ldots, \pi_p ) \in \TPath{P}{D}{\alpha}{\beta}$ with $\pi_j \in U_1$ for some $j$.  Observe that $2 \leqslant j \leqslant p-1$ because $\pi_1 = \alpha$ and $\pi_p = \beta$ are in $V_2$ by assumption.  However, no diagonal in $D$ has an end point in $U_1$, so the diagonals $\{ \pi_{ j-1 },\pi_j \}$ and $\{ \pi_j,\pi_{ j+1 } \}$ are not in $D$, contradicting Definition \ref{def:T-paths_and_T-path_formula}(iii).
\end{proof}

Lemma \ref{lem:14} has the following immediate consequence.

\begin{Lemma}
\label{lem:T-path_formula_on_P2}
If $f : \diag(P) \xrightarrow{} K$ satisfies the $T$-path formula with respect to $D$, then $f\Big|_{ \diag( P_2 ) }$ satisfies the $T$-path formula with respect to $D_2$, see Definition \ref{def:T-paths_and_T-path_formula}.
\end{Lemma}

\begin{Lemma}
\label{lem:10}
Assume that $\alpha \in U_1$, $\beta \in U_2$ and let $\pi = ( \pi_1, \ldots, \pi_p ) \in \TPath{P}{D}{\alpha}{\beta}^{ (\alpha,\zeta,\neq \eta ) }$ be given.  Then $\eta < \pi_i \leqslant \zeta$ for $i \geqslant 2$.
\end{Lemma}

\begin{proof}
Observe that $p \geqslant 3$.  By Definition \ref{def:T-paths_and_T-path_formula}(iii) the diagonal $\{ \pi_2,\pi_3 \}$ is in $D$, so both end points are in $V_2$.  
Since $\pi_2 = \zeta$ and $\pi_3 \neq \eta$, it follows that $\eta < \pi_3 < \zeta$, see Figure \ref{fig:10}.
\begin{figure}
\begingroup
\[
  \begin{tikzpicture}[auto]
    \node[name=s, shape=regular polygon, regular polygon sides=20, minimum size=6cm, draw] {}; 
    \draw (3.5+8.5*360/20:3.65cm) node {$\pi_2 = \zeta$};
    \draw (-3+11.5*360/20:3.70cm) node {$\pi_1 = \alpha$};
    \draw (15.5*360/20:3.25cm) node {$\eta$};
    \draw (17.5*360/20:3.25cm) node {$\;\;\; \pi_3$};
    \draw (21.5*360/20:3.25cm) node {$\beta$};
    \draw (8.5*360/20:3cm) to (15.5*360/20:3cm);
    \draw[dashed] (11.5*360/20:3cm) to (8.5*360/20:3cm);
    \draw[dashed] (8.5*360/20:3cm) to (17.5*360/20:3cm);
    \draw[dotted] (11.5*360/20:3cm) to (21.5*360/20:3cm);
    \draw (-1.2,-1.8) node {$P_1$};
    \draw (-0.2,0.8) node {$P_2$};
  \end{tikzpicture} 
\]
\endgroup
\caption{The $T$-path $\pi$ in the proof of Lemma \ref{lem:10} starts as shown (dashed).}
\label{fig:10}
\end{figure}
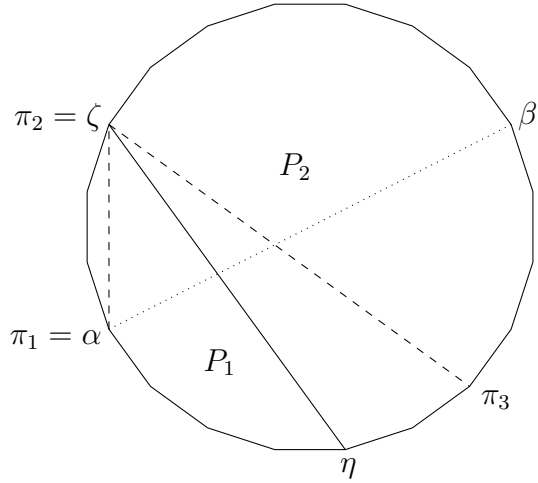

By Definition \ref{def:T-paths_and_T-path_formula}(iii), the diagonals $\{ \pi_{ 2j },\pi_{ 2j+1 } \}$ are in $D$, and they cross the diagonal $\{ \alpha,\beta \}$ at pairwise different points which progress monotonically from $\alpha$ to $\beta$.  It follows that each $\{ \pi_{ 2j },\pi_{ 2j+1 } \}$ has at least one end point, which we will denote $\gamma$, satisfying $\pi_3 \leqslant \gamma \leqslant \pi_2$ whence $\eta < \gamma \leqslant \zeta$, see Figure \ref{fig:10}.  It is enough to show that the other end point of $\{ \pi_{ 2j },\pi_{ 2j+1 } \}$, which we will denote $\delta$, also satisfies $\pi_3 \leqslant \delta \leqslant \pi_2$. 

Assuming the opposite, we have $\pi_2 < \delta < \pi_3$.  If we had $\pi_3 < \gamma < \pi_2$, then $\{ \pi_{ 2j },\pi_{ 2j+1 } \} = \{ \gamma,\delta \}$ would cross $\{ \pi_2,\pi_3 \}$ which is in $D$, contradicting Definition \ref{def:T-paths_and_T-path_formula}(ii).  If we had $\gamma$ equal to $\pi_3$ or $\pi_2$, then $\{ \pi_{ 2j },\pi_{ 2j+1 } \} = \{ \gamma,\delta \}$ would cross $\{ \alpha,\beta \}$ closer to $\alpha$ than $\{ \pi_2,\pi_3 \}$ does, contradicting Definition \ref{def:T-paths_and_T-path_formula}(iii).  
\end{proof}

\begin{Lemma}
\label{lem:U1_to_U2}
Assume that $\alpha \in U_1$, $\beta \in U_2$ and let $\pi = ( \pi_1, \ldots, \pi_p ) \in \TPath{P}{D}{\alpha}{\beta}$ be given.  Then $\pi_2 \in \{ \zeta,\eta \}$.  
\end{Lemma}

\begin{proof}
Since $\pi_1 = \alpha \in U_1$, we must have $\pi_2 \in V_1$, since the alternative $\pi_2 \in U_2$ would imply that $\{ \pi_1,\pi_2 \}$ crossed $\{ \zeta,\eta \}$, contradicting Definition \ref{def:T-paths_and_T-path_formula}(ii) since $\{ \zeta,\eta \}$ is in $D$.  Note that $\pi_2 \in V_1$ implies $\pi_2 \neq \beta$ so $p \geqslant 3$.  Assuming $\pi_2 \not\in \{ \zeta,\eta \}$, we have $\pi_2 \in U_1$ whence $\{ \pi_2,\pi_3 \}$ is not in $D$, contradicting Definition \ref{def:T-paths_and_T-path_formula}(iii).
\end{proof}

\begin{Lemma}
\label{lem:17_18}
Assume that $\alpha \in U_1$, $\beta \in U_2$ and let $f : \diag(P) \xrightarrow{} K$ be a map.
\begin{enumerate}
\setlength\itemsep{4pt}

  \item  There is a bijection
\[
  \TPath{P}{D}{\alpha}{\beta}^{ ( \alpha,\zeta,\neq \eta ) }
  \xrightarrow{R} 
  \TPath{P}{D}{\eta}{\beta}^{ ( \eta,\zeta ) }
\]
given by $R( \alpha,\zeta,\pi_3, \ldots, \pi_p ) = ( \eta,\zeta,\pi_3, \ldots, \pi_p )$, which satisfies
\begin{equation}
\label{equ:R}
  \frac{ f( \alpha,\zeta ) }{ f( \eta,\zeta ) }
  f\big( R( \pi ) \big)
  =
  f( \pi ).
\end{equation}

  \item  There is a bijection
\[
  \TPath{P}{D}{\alpha}{\beta}^{ ( \alpha,\zeta,\eta ) }
  \xrightarrow{S} 
  \TPath{P}{D}{\eta}{\beta}^{ ( \eta,\neq \zeta ) }
\]
given by $S( \alpha,\zeta,\eta,\pi_4, \ldots, \pi_p ) = ( \eta,\pi_4, \ldots, \pi_p )$, which satisfies
\begin{equation}
\label{equ:S}
  \frac{ f( \alpha,\zeta ) }{ f( \zeta,\eta ) }
  f\big( S( \pi ) \big)
  =
  f( \pi ).
\end{equation}

\end{enumerate}
\end{Lemma}

\begin{proof}
Parts (i) and (ii) of the lemma can be proved by similar methods, and we only show the proof of (i).

Step 1:  $R$ maps into $\TPath{P}{D}{\eta}{\beta}^{ ( \eta,\zeta ) }$:  Let $\pi = ( \alpha,\zeta,\pi_3, \ldots, \pi_p ) \in \TPath{P}{D}{\alpha}{\beta}^{ (\alpha,\zeta,\neq \eta ) }$ be given.  We must show that $\rho := R( \pi ) = ( \eta,\zeta,\pi_3, \ldots, \pi_p )$ is in $\TPath{P}{D}{\eta}{\beta}^{ ( \eta,\zeta ) }$.

The first two vertices of $\rho$ are indeed $\rho_1 = \eta$, $\rho_2 = \zeta$.  It is hence enough to prove that $\rho$ satisfies Definition \ref{def:T-paths_and_T-path_formula}, parts (i)-(iii).  Part (ii) for $\rho$ is clear from part (ii) for $\pi$ combined with $\{ \eta,\zeta \} \in D$.  

Definition \ref{def:T-paths_and_T-path_formula}(i) holds for $\rho$ because it holds for $\pi$ while $\rho_i = \pi_i \neq \eta$ for $i \geqslant 3$.  Indeed, Lemma \ref{lem:10} says
\[
  \eta < \pi_i \leqslant \zeta
  \;\;\mbox{for}\;\;
  i \geqslant 2.
\]
This inequality also implies Definition \ref{def:T-paths_and_T-path_formula}(iii) for $\rho$, because it means that since the diagonals $\{ \rho_{ 2j },\rho_{ 2j+1 } \} = \{ \pi_{ 2j },\pi_{ 2j+1 } \}$ cross $\{ \alpha,\beta \}$ at pairwise different points which progress monotonically in the direction from $\alpha$ to $\beta$, they cross $\{ \eta,\beta \}$ at pairwise different points which progress monotonically in the direction from $\eta$ to $\beta$, see Figure \ref{fig:17_18a}.
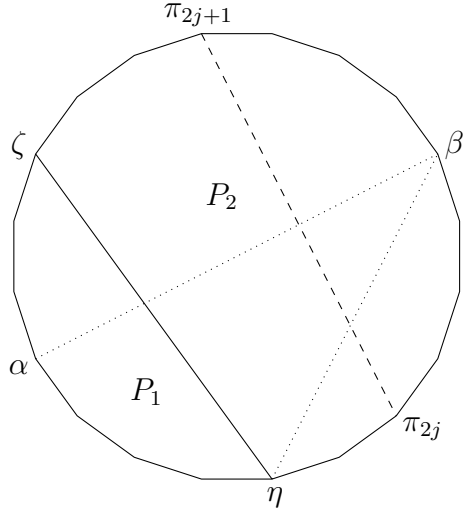
\begin{figure}
\begingroup
\[
  \begin{tikzpicture}[auto]
    \node[name=s, shape=regular polygon, regular polygon sides=20, minimum size=6cm, draw] {}; 
    \draw (8.5*360/20:3.25cm) node {$\zeta$};
    \draw (11.5*360/20:3.25cm) node {$\alpha$};
    \draw (15.5*360/20:3.25cm) node {$\eta$};
    \draw (17.5*360/20:3.25cm) node {$\;\;\; \pi_{ 2j }$};
    \draw (21.5*360/20:3.25cm) node {$\beta$};
    \draw (5.5*360/20:3.25cm) node {$\pi_{ 2j+1 }$};
    \draw (8.5*360/20:3cm) to (15.5*360/20:3cm);
    \draw[dotted] (11.5*360/20:3cm) to (21.5*360/20:3cm);
    \draw[dotted] (15.5*360/20:3cm) to (21.5*360/20:3cm);
    \draw[dashed] (5.5*360/20:3cm) to (17.5*360/20:3cm);
    \draw (-1.2,-1.8) node {$P_1$};
    \draw (-0.2,0.8) node {$P_2$};
  \end{tikzpicture} 
\]
\endgroup
\caption{The figure concerns Lemma \ref{lem:17_18}.  In Step 1 of the proof the diagonals $\{ \pi_{ 2j },\pi_{ 2j+1 } \}$ cross $\{ \alpha,\beta \}$, and since $\eta < \pi_i \leqslant \zeta$ for $i \geqslant 2$ they also cross $\{ \eta,\beta \}$.  In Step 2 of the proof the diagonals $\{ \rho_{ 2j },\rho_{ 2j+1 } \}$ cross $\{ \eta,\beta \}$, and since $\eta \leqslant \rho_i \leqslant \zeta$ for $i \geqslant 1$ they also cross $\{ \alpha,\beta \}$.}
\label{fig:17_18a}
\end{figure}

Step 2:  $R$ is surjective:  Let $\rho = ( \eta,\zeta,\rho_3, \ldots, \rho_p ) \in \TPath{P}{D}{\eta}{\beta}^{ ( \eta,\zeta ) }$ be given.  We will show that $\pi := ( \alpha,\zeta,\rho_3, \ldots, \rho_p )$ is in $\TPath{P}{D}{\alpha}{\beta}^{ (\alpha,\zeta,\neq \eta ) }$, whence $\rho = R( \pi )$ is clear.  

The first three vertices of $\pi$ are indeed $\pi_1 = \alpha$, $\pi_2 = \zeta$, $\pi_3 \neq \eta$, since $\pi_3 = \rho_3 \neq \eta$ because $\rho \neq ( \eta,\zeta,\eta, \ldots )$ by Definition \ref{def:T-paths_and_T-path_formula}(i).  It is hence enough to prove that $\pi$ satisfies Definition \ref{def:T-paths_and_T-path_formula}, parts (i)-(iii).  Part (ii) for $\pi$ is clear from part (ii) for $\rho$ combined with the observation that the internal diagonal $\{ \alpha,\zeta \}$ of $P_1$ does not cross any diagonal in $D$.

Definition \ref{def:T-paths_and_T-path_formula}(i) holds for $\pi$ because it holds for $\rho$ while $\pi_i = \rho_i \neq \alpha$ for $i \geqslant 3$.  Indeed, Lemma \ref{lem:14} implies
\[
  \eta \leqslant \rho_i \leqslant \zeta
  \;\;\mbox{for}\;\;
  i \geqslant 1.
\]
This inequality also implies Definition \ref{def:T-paths_and_T-path_formula}(iii) for $\pi$, because it means that since the diagonals $\{ \pi_{ 2j },\pi_{ 2j+1 } \} = \{ \rho_{ 2j },\rho_{ 2j+1 } \}$ cross $\{ \eta,\beta \}$ at pairwise different points which progress monotonically from $\eta$ to $\beta$, they cross $\{ \alpha,\beta \}$ at pairwise different points which progress monotonically from $\alpha$ to $\beta$, see Figure \ref{fig:17_18a}.

Step 3:  $R$ is injective: This is clear from the formula defining $R$.

Step 4:  Equation \eqref{equ:R}:  This follows immediately by combining Equation \eqref{equ:f_on_T-paths} with the definition of $R$.  
\end{proof}

\begin{Lemma}
\label{lem:towards_Thm_A}
Assume that $\alpha \in U_1$, $\beta \in U_2$ and let $f : \diag(P) \xrightarrow{} K$ be a map such that $f\Big|_{ \diag( P_2 ) }$ satisfies the $T$-path formula with respect to $D_2$, see Definition \ref{def:T-paths_and_T-path_formula}.  Then
\begin{equation}
\label{equ:towards_Thm_A5}
  f( \zeta,\eta )^{ -1 }\big[ f( \alpha,\zeta )f( \eta,\beta ) + f( \alpha,\eta )f( \zeta,\beta ) \big] = \sum_{ \pi \in \TPath{P}{D}{\alpha}{\beta} } f( \pi ).
\end{equation}
\end{Lemma}

\begin{proof}
Lemma \ref{lem:14} gives the first of the following equalities, and Lemma \ref{lem:17_18}(i) gives the second and third equalities.
\begin{align}
\nonumber
  \frac{ f( \alpha,\zeta ) }{ f( \eta,\zeta ) }\sum_{ \rho \in \TPath{P_2}{D_2}{\eta}{\beta}^{ ( \eta,\zeta ) } }f( \rho )
  & = \frac{ f( \alpha,\zeta ) }{ f( \eta,\zeta ) }\sum_{ \rho \in \TPath{P}{D}{\eta}{\beta}^{ ( \eta,\zeta ) } }f( \rho ) \\
\nonumber
  & = \frac{ f( \alpha,\zeta ) }{ f( \eta,\zeta ) }\sum_{ \pi \in \TPath{P}{D}{\alpha}{\beta}^{ ( \alpha,\zeta,\neq \eta ) } }f\big( R( \pi ) \big) \\
\label{equ:towards_Thm_A1}
  & = \sum_{ \pi \in \TPath{P}{D}{\alpha}{\beta}^{ ( \alpha,\zeta,\neq \eta ) } } f( \pi ).
\end{align}
Similarly, Lemmas \ref{lem:14} and \ref{lem:17_18}(ii) give the following equalities.
\begin{align}
\nonumber
  \frac{ f( \alpha,\zeta ) }{ f( \zeta,\eta ) }\sum_{ \sigma \in \TPath{P_2}{D_2}{\eta}{\beta}^{ ( \eta,\neq \zeta ) } }f( \sigma )
  & = \frac{ f( \alpha,\zeta ) }{ f( \zeta,\eta ) }\sum_{ \sigma \in \TPath{P}{D}{\eta}{\beta}^{ ( \eta,\neq \zeta ) } }f( \sigma ) \\
\nonumber
  & = \frac{ f( \alpha,\zeta ) }{ f( \zeta,\eta ) }\sum_{ \pi \in \TPath{P}{D}{\alpha}{\beta}^{ ( \alpha,\zeta,\eta ) } }f\big( S( \pi ) \big) \\
\label{equ:towards_Thm_A2}
  & = \sum_{ \pi \in \TPath{P}{D}{\alpha}{\beta}^{ ( \alpha,\zeta,\eta ) } } f( \pi ).
\end{align}
The sum of Equations \eqref{equ:towards_Thm_A1} and \eqref{equ:towards_Thm_A2} is
\[
  \frac{ f( \alpha,\zeta ) }{ f( \zeta,\eta ) }\sum_{ \tau \in \TPath{P_2}{D_2}{\eta}{\beta} }f( \tau ) = \sum_{ \pi \in \TPath{P}{D}{\alpha}{\beta}^{ ( \alpha,\zeta ) } } f( \pi ).
\]
Since $f\Big|_{ \diag( P_2 ) }$ satisfies the $T$-path formula with respect to $D_2$, this reads
\begin{equation}
\label{equ:towards_Thm_A3}
  \frac{ f( \alpha,\zeta ) }{ f( \zeta,\eta ) }f( \eta,\beta ) = \sum_{ \pi \in \TPath{P}{D}{\alpha}{\beta}^{ ( \alpha,\zeta ) } } f( \pi ).
\end{equation}
By symmetry,
\begin{equation}
\label{equ:towards_Thm_A4}
  \frac{ f( \alpha,\eta ) }{ f( \eta,\zeta ) }f( \zeta,\beta ) = \sum_{ \pi \in \TPath{P}{D}{\alpha}{\beta}^{ ( \alpha,\eta ) } } f( \pi ).
\end{equation}
Lemma \ref{lem:U1_to_U2} implies that the sum of Equations \eqref{equ:towards_Thm_A3} and \eqref{equ:towards_Thm_A4} is Equation \eqref{equ:towards_Thm_A5}.
\end{proof}

\section{Proof of Theorem \ref{thm:A}}
\label{sec:Thm_A}

\begin{proof}[Proof of Theorem \ref{thm:A}]
Recall that the $T$-path formula was introduced in Definition \ref{def:T-paths_and_T-path_formula}.  It will be invoked a number of times in the proof.

``If'':
Assume that $f$ satisfies the $T$-path formula with respect to $D$.  We will prove by induction on $\ell := |D|$ that $f$ is a weak frieze with respect to $D$.

If $\ell = 0$ then $D = \emptyset$ and it is immediate from Definition \ref{def:friezes_and_weak_friezes}(ii) that 
$f$ is a weak frieze with respect to $D$.

If $\ell \geqslant 1$, then we can assume to be in the situation of Setup \ref{set:ear} and use the notation introduced there.  Setting $m=1$, $d_1=d$, $D_1 = \emptyset$ then puts us in the situation of Theorem \ref{thm:B}.

It is immediate from Definition \ref{def:friezes_and_weak_friezes}(ii) that $f\Big|_{ \diag( P_1 ) }$ is a weak frieze with respect to $D_1 = \emptyset$.  Lemma \ref{lem:T-path_formula_on_P2} says that $f\Big|_{ \diag( P_2 ) }$ satisfies the $T$-path formula with respect to $D_2$.  Since $D = \{ d \} \cupdot D_2$ we have $|D_2| = \ell-1$, so by induction $f\Big|_{ \diag( P_2 ) }$ is a weak frieze with respect to $D_2$.

By Theorem \ref{thm:B} there exists a unique weak frieze $\widetilde{f} : \diag( P ) \xrightarrow{} K$ with respect to $D$ such that
\begin{equation}
\label{equ:restriction}
  \widetilde{f}\,\Big|_{ \diag( P_i ) } = f\Big|_{ \diag( P_i ) }
\end{equation}
for $i \in \{ 1,2 \}$.  It is sufficient to show $\widetilde{f} = f$, that is
\begin{equation}
\label{equ:f_equality}
  \widetilde{f}( \alpha,\beta ) = f( \alpha,\beta )
\end{equation}
for vertices $\alpha \neq \beta$ of $P$.  There are four cases to consider.

Case 1:  $\alpha,\beta \in V_1$.  Then Equation \eqref{equ:f_equality} follows from Equation \eqref{equ:restriction}.

Case 2: $\alpha,\beta \in V_2$.  Same argument as in Case 1.

Case 3:  $\alpha \in U_1$, $\beta \in U_2$.  Since $f\Big|_{ \diag( P_2 ) }$ satisfies the $T$-path formula with respect to $D_2$, Lemma \ref{lem:towards_Thm_A} gives Equation \eqref{equ:towards_Thm_A5}.  Since $f$ satisfies the $T$-path formula with respect to $D$, the equation reads
\[
  f( \zeta,\eta )^{ -1 }\big[ f( \alpha,\zeta )f( \eta,\beta ) + f( \alpha,\eta )f( \zeta,\beta ) \big] = f( \alpha,\beta ).
\]
Equation \eqref{equ:restriction} implies that this reads
\begin{equation}
\label{equ:implies_f_equality}
  \widetilde{f}( \zeta,\eta )^{ -1 }\big[ \widetilde{f}( \alpha,\zeta )\widetilde{f}( \eta,\beta ) + \widetilde{f}( \alpha,\eta )\widetilde{f}( \zeta,\beta ) \big] = f( \alpha,\beta ).
\end{equation}
But $\widetilde{f}$ is a weak frieze with respect to $D$ and $d = \{ \zeta,\eta \}$ is in $D$, so Equation \eqref{equ:implies_f_equality} becomes Equation \eqref{equ:f_equality}.

Case 4:  $\alpha \in U_2$, $\beta \in U_1$.  This reduces to Case 3 because each side of Equation \eqref{equ:f_equality} is symmetric in $\alpha$ and $\beta$ since $f$ and $\widetilde{f}$ are defined on $\diag(P)$.

``Only if'':
Assume that $f$ is a weak frieze with respect to $D$.  We will prove by induction on $\ell := |D|$ that $f$ satisfies the $T$-path formula with respect to $D$.

If $\ell = 0$ then $D = \emptyset$.  It follows that if $\alpha \neq \beta$ are vertices of $P$, then $\{ \alpha,\beta \}$ crosses no diagonal in $D$, so Equation \eqref{equ:the_T-path_formula} holds by Lemma \ref{lem:trivial_T-paths}(ii).  That is, $f$ satisfies the $T$-path formula with respect to $D$.   

If $\ell \geqslant 1$, then we can assume to be in the situation of Setup \ref{set:ear} and use the notation introduced there.  It is clear from Definition \ref{def:friezes_and_weak_friezes}(ii) that $f\Big|_{ \diag( P_2 ) }$ is a weak frieze with respect to $D_2$.  Since $D = \{ d \} \cupdot D_2$ we have $|D_2| = \ell-1$, so by induction $f\Big|_{ \diag( P_2 ) }$ satisfies the $T$-path formula with respect to $D_2$.  To show that $f$ satisfies the $T$-path formula with respect to $D$, let $\alpha \neq \beta$ be vertices of $P$.  We must show that Equation \eqref{equ:the_T-path_formula} holds, and there are four cases to consider.

Case 1:  $\alpha, \beta \in V_1$.  Then the diagonal $\{ \alpha,\beta \}$ crosses no diagonal in $D$, so Equation \eqref{equ:the_T-path_formula} holds by Lemma \ref{lem:trivial_T-paths}(ii).

Case 2:  $\alpha, \beta \in V_2$.  Then
\[
  f( \alpha,\beta ) = \sum_{ \pi \in \TPath{P_2}{D_2}{\alpha}{\beta} } f( \pi ) = \sum_{ \pi \in \TPath{P}{D}{\alpha}{\beta} } f( \pi )
\]
as desired.  The first equality holds because $f\Big|_{ \diag( P_2 ) }$ satisfies the $T$-path formula with respect to $D_2$, and the second equality is by Lemma \ref{lem:14}.

Case 3:  $\alpha \in U_1$, $\beta \in U_2$.  Since $f\Big|_{ \diag( P_2 ) }$ satisfies the $T$-path formula with respect to $D_2$, Lemma \ref{lem:towards_Thm_A} applies and gives Equation \eqref{equ:towards_Thm_A5}.  Since $d = \{ \zeta,\eta \}$ is in $D$ and crosses $\{ \alpha,\beta \}$ while $f$ is a weak frieze with respect to $D$, Equation \eqref{equ:towards_Thm_A5} becomes Equation \eqref{equ:the_T-path_formula} as desired.

Case 4:  $\alpha \in U_2$, $\beta \in U_1$.  This reduces to Case 3 because each side of Equation \eqref{equ:the_T-path_formula} is symmetric in $\alpha$ and $\beta$.  The left hand side is symmetric since $f$ is defined on $\diag(P)$, and the right hand side is symmetric by Lemma \ref{lem:inversion}(ii).
\end{proof}

\medskip
\noindent
{\bf Acknowledgement.}
This work was supported by EPSRC grant EP/P016014/1 ``Higher Dimensional Homological Algebra''.


\begin{thebibliography}{19}

\bibitem{BR}  A.\ Berenstein and V.\ Retakh, \href{https://doi.org/10.1016/j.aim.2018.02.014}{\it Noncommutative marked surfaces}, Adv.\ Math.\ {\bf 328} (2018), 1010--1087.


\bibitem{BHJ}  C.\ Bessenrodt, T.\ Holm, and P.\ J\o rgensen, \href{https://doi.org/10.1016/j.jcta.2013.11.003}{\it Generalized frieze pattern determinants and higher angulations of polygons}, J.\ Combin.\ Theory Ser.\ A {\bf 123} (2014), 30--42.

\bibitem{CHJ}  M.\ Cuntz, T.\ Holm, and P.\ J\o rgensen, \href{https://doi.org/10.1017/fms.2020.13}{\it Frieze patterns with coefficients}, Forum Math.\ Sigma {\bf 8} (2020), e17.

\bibitem{Cox}  H.\ S.\ M.\ Coxeter, {\it Frieze patterns}, Acta Arith.\ {\bf 18} (1971), 297--310.



\bibitem{FZ}  S.\ Fomin and A.\ Zelevinsky, \href{https://doi.org/10.1090/S0894-0347-01-00385-X}{\it Cluster algebras I: Foundations}, J.\ Amer.\ Math.\ Soc.\ {\bf 15} (2002), 497--529.

\bibitem{GHKK}  M.\ Gross, P.\ Hacking, S.\ Keel, and M.\ Kontsevich, \href{https://doi.org/10.1090/jams/890}{\it Canonical bases for cluster algebras}, J.\ Amer.\ Math.\ Soc.\ {\bf 31} (2018), 497--608.

\bibitem{GM}  E.\ Gunawan and G.\ Musiker, \href{http://dx.doi.org/10.3842/SIGMA.2015.060}{\it T-path formula and atomic bases for cluster algebras of type D}, SIGMA Symmetry Integrability Geom.\ Methods Appl.\ {\bf 11} (2015), Paper 060.

\bibitem{G}  L.\ Guo, \href{https://doi.org/10.1093/imrn/rns176}{\it On tropical friezes associated with Dynkin diagrams}, Internat.\ Math.\ Res.\ Notices {\bf 2013}, no.\ 18, 4243--4284. 

\bibitem{HJ} T.\ Holm and P.\ J\o rgensen, \href{https://doi.org/10.1093/imrn/rny020}{\it A $p$-angulated generalisation of Conway and Coxeter's theorem on frieze patterns}, Internat.\ Math.\ Res.\ Notices {\bf 2020}, no.\ 1, 71--90.


\bibitem{LS}  K.\ Lee and R.\ Schiffler, \href{https://doi.org/10.4007/annals.2015.182.1.2}{\it Positivity for cluster algebras}, Ann.\ of Math.\ (2) {\bf 182} (2015), 73--125.


\bibitem{M}  S.\ Morier-Genoud, \href{https://doi.org/10.1112/blms/bdv070}{\it Coxeter's frieze patterns at the crossroads of algebra, geometry and combinatorics}, Bull.\ London Math.\ Soc.\ {\bf 47} (2015), 895--938.


\bibitem{MSW}  G.\ Musiker, R.\ Schiffler, and L.\ Williams, \href{https://doi.org/10.1016/j.aim.2011.04.018}{\it Positivity for cluster algebras from surfaces}, Adv.\ Math.\ {\bf 227} (2011), 2241--2308.

\bibitem{Propp}  J.\ Propp, \href{https://arxiv.org/abs/math/0511633v4}{The combinatorics of frieze patterns and Markoff numbers}, preprint (2005).  {\tt arXiv:0511633v4}

\bibitem{Sch1}  R.\ Schiffler, \href{https://doi.org/10.37236/788}{\it A cluster expansion formula ($A_n$ case)}, Electron.\ J.\ Combin.\ {\bf 15} (2008), \#R64.

\bibitem{Sch2}  R.\ Schiffler, \href{https://doi.org/10.1016/j.aim.2009.10.015}{\it On cluster algebras arising from unpunctured surfaces II}, Adv.\ Math.\ {\bf 223} (2010), 1885--1923.

\bibitem{ST}  R.\ Schiffler and H.\ Thomas, \href{https://doi.org/10.1093/imrn/rnp047}{\it 
On cluster algebras arising from unpunctured surfaces}, 
Internat.\ Math.\ Res.\ Notices {\bf 2009}, no.\ 17, 3160--3189.

\end{thebibliography}
\end{document}